%% file: nes.tex
\def\dofigures{11}

\documentclass[10pt]{article}

\def\print{1}
\def\myfmt{\print}%
\input{elle}

\usepackage{calc}
\usepackage{color}
\usepackage{amsmath}
\usepackage{amssymb}
\usepackage{amsthm}
\usepackage[sans]{dsfont}

\usepackage{textcomp}
\usepackage{mathrsfs}

\usepackage{bbding}

\usepackage{mathtools}

\usepackage{fancyhdr}
\usepackage{supertabular}

\usepackage{verbatim}

\input{mylabels}

\usepackage[dvips]{graphicx}

\definecolor{orange}{rgb}{1,0.5,0}
\definecolor{brown}{rgb}{0.5,0.25,0}
\definecolor{violet}{rgb}{0.25,0.0,0.65}
\definecolor{darkgreen}{rgb}{0,0.5,0}
\definecolor{darkred}{rgb}{0.65,0,0}
\definecolor{lightgrey}{rgb}{0.5,0.5,0.5}
\definecolor{midgreen}{rgb}{0,0.8,0}

\if\myfmt\print%

\newcommand{\uerr}[1]{}
\newcommand{\ustrat}[1]{}
\newcommand{\urem}[1]{}
\newcommand{\uimp}[1]{}
\else%

\newcommand{\uerr}[1]{{\color{red}(ERROR: #1)}}
\newcommand{\ustrat}[1]{{\color{darkgreen}(STRAT: #1)}}
\newcommand{\urem}[1]{{\color{green}{(REM: #1)}}}
\newcommand{\uimp}[1]{{\color{magenta}{(IMP: #1)}}}
\fi%

\if\print%
\newcommand{\myclearpage}{}
\else%
\newcommand{\myclearpage}{\clearpage}
\fi%

\newcommand{\vxi}{\vec\xi}
\newcommand{\vxr}{\vxi_r}
\newcommand{\vxs}{\vxi}
\newcommand{\vxt}{\vxi_T}
\newcommand{\vn}{\vec n}
\newcommand{\vt}{\vec t}

\newcommand{\vz}{\vec z}
\newcommand{\vI}{\vv_I}

\newcommand{\gv}{g_{\vv}}

\newcommand{\suso}{\Psi}

\def\pmnewlabel#1#2{\expandafter\def\csname pmref-#1\endcsname{#2}}
\def\pmref#1{\csname pmref-#1\endcsname}
\def\pmeqref#1{(\csname pmref-#1\endcsname)}

\if\myfmt\print\else%
\fi%

\begin{document}

\if\myfmt\print\else%
\pagestyle{empty}
\thispagestyle{empty}
\fi%

\if\myfmt\print\else%
\tableofcontents%
\myclearpage%
\fi%

\title{Non-existence of strong regular reflections in self-similar potential flow\footnote{This material is based upon work supported by the
National Science Foundation under Grant No.\ NSF DMS-0907074}}
\author{Volker Elling}
\date{} 
\maketitle

\begin{abstract}
    We consider shock reflection which has a well-known local non-uniqueness: 
    the reflected shock can be either of two choices, called weak and strong. 
    We consider cases where existence of a global solution with weak reflected shock has been proven,
    for compressible potential flow. 
    If there was a global strong-shock solution as well, then potential flow would be ill-posed.
    However, we prove non-existence of strong-shock analogues in a natural class of candidates.
\end{abstract}

\if\myfmt\print\else%
\pagestyle{empty}
\thispagestyle{empty}
\fi%

\if\myfmt\print\else%
\parindent=0cm%
\parskip=\baselineskip%
\fi%

\section{Introduction}

\mylabel{section:refl}

In compressible inviscid flow, a shock wave is a surface across which pressure, temperature and normal velocity are discontinuous
(while tangential velocity is continuous). Shock waves form especially in supersonic or transonic flow near solid surfaces,
and even in the absence of boundaries in finite time from smooth initial data. 

Reflection of shock waves, by interaction with each other or with solid surfaces, is a classical problem of compressible flow. 
It has been studied extensively by Ernst Mach \cite{mach-wosyka,krehl-geest} and John von Neumann \cite{neumann-1943}, among others.

\if\dofigures%
\begin{figure}[h]
    \input{classrr.pstex_t}
    \caption{Center left: classical RR (known solution); center right: analogue with strong-type shock (we prove non-existence);
        right: single MR.}
    \mylabel{fig:classrr}
\end{figure}
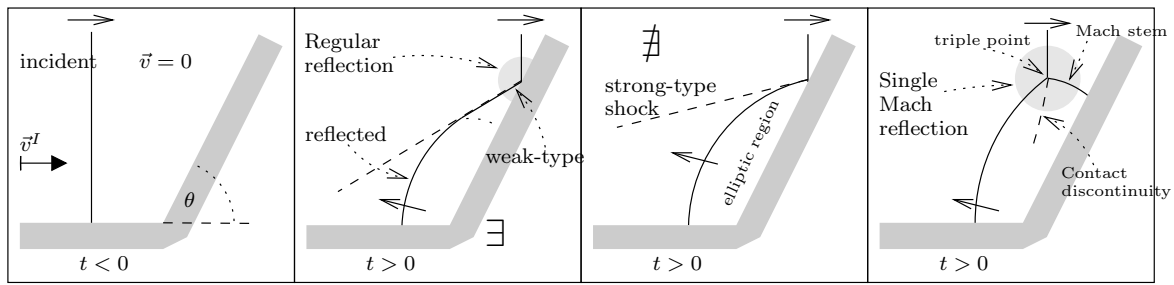
\fi%

Recently there have been some breakthroughs in constructive existence proofs for particular shock reflections as
solutions of 2d compressible potential flow. 
First, \cite{chen-feldman-selfsim-journal} 
(see also \cite{canic-keyfitz-lieberman,yuxi-zheng-rref,elling-rrefl,bae-chen-feldman}) constructed a global solution
of a problem we call ``classical regular reflection'' here (see Figure \myref{fig:classrr} center left).
A shock wave (``incident'') approaches along a solid wall (Figure \myref{fig:classrr} left), reaching the wall corner at time $t=0$. 
For $t>0$ the incident shock continues along the ramp, while a curved shock (``reflected'') travels back upstream.
For some values of the parameters (wall angle $\theta$, upstream Mach number $|\vv_u|/c_u$), 
the two shocks meet in a single point (``reflection point'') on the ramp,
a local configuration called \defm{regular reflection} (\defm{RR}; Figure \myref{fig:classrr} center left and right). 
At other parameters \defm{Mach reflection} (\defm{MR}) is observed,
for  example \defm{single Mach reflection} (\defm{SMR}; Figure \myref{fig:classrr} right) 
where the two shocks meet in a \defm{triple point} away from the wall
with a contact discontinuity and a third shock (``Mach stem'') that connects to the wall (this pattern is not possible in
potential flow).

\if\dofigures%
\begin{figure}[h]
    \input{superwedge.pstex_t}
    \caption{Left: constant supersonic velocity (initial data); center: weak-type solution (known solution); right: analogue
        with strong-type shock (we prove non-existence).}
    \mylabel{fig:superwedge}
\end{figure}
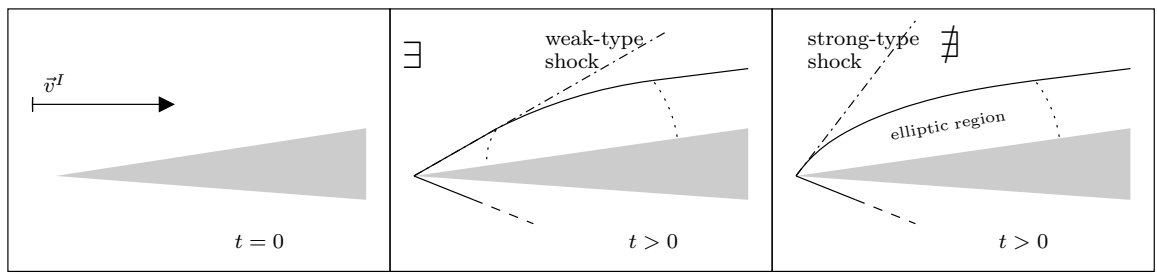
\fi%
In related work, Liu and the author \cite{elling-liu-pmeyer} considered supersonic flow onto a solid wedge 
(see Figure \myref{fig:superwedge}). At time $t=0$, the fluid state is the same in every point,
with sufficiently large supersonic velocity $\vv^I$ pointing in the right horizontal direction (Figure \myref{fig:superwedge} left).
For $t>0$, the fluid impinging on the wedge produces a shock wave (Figure \myref{fig:superwedge} center). 
Near infinity the shock wave is straight and parallel to the wedge 
(due to finite speed of sound it cannot ``see'' the differnence between a wedge and an infinite line wall); 
closer to the tip the shock curves and meets the wedge tip.

Both classical regular reflection and the supersonic wedge yield \defm{self-similar} flow: 
density, temperature and velocity are functions of the similarity coordinate $\vxi=\vec x/t$ alone. 
Patterns grow proportional to time $t$; the flow at time $t_2$ is obtained from the flow at time $t_1$ 
by a dilation by factor $t_2/t_1$.
$t\downarrow 0$ corresponds to ``zooming infinitely far away''
whereas $t\uparrow+\infty$ is like ``zooming into the origin'' or ``scaling up''. 

While \cite{chen-feldman-selfsim-journal,elling-rrefl} and \cite{elling-liu-pmeyer} prove existence of certain flows, 
as solutions of 2d compressible potential flow, for a range of wedge angles and upstream Mach numbers, 
they do \emph{not} prove uniqueness. The latter is an important question, due to a problematic feature of local RR:
take the point of reference of an observer located in the reflection point at all times.
Focus on a small neighbourhood of the reflection point (``local regular reflection''; Figure \myref{fig:locrr-left} left). 
The problem parameters already determine the angle between incident shock and wall as well as
the fluid state (velocity $\vv$, density, temperature) in the 1- and 2-sector.
The reflected shock is a steady shock passing through the reflection point, 
with 2-sector fluid state as upstream data, that results in a 3-sector velocity $\vv_3$ \emph{parallel} to the wall.

\begin{figure}[h]
\input{locrr.pstex_t}%
\input{polar.pstex_t}%
\caption{Left: local RR. Right: fixed $\vv_2$; each steady shock produces
one $\vv_3$ on the curve (shock polar, symmetric across $\vv_2$; shock normal $\parallel\vv_2-\vv_3$). 
For $|\tau|<\tau_*$, three shocks satisfy $\tau=\measuredangle(\vv_2,\vv_3)$: strong-type (S),
weak-type (W) and expansion (U; unphysical). W are transonic right of $+$, supersonic left.}
\mylabel{fig:locrr-left}%
\mylabel{fig:spolar-right}%
\end{figure}
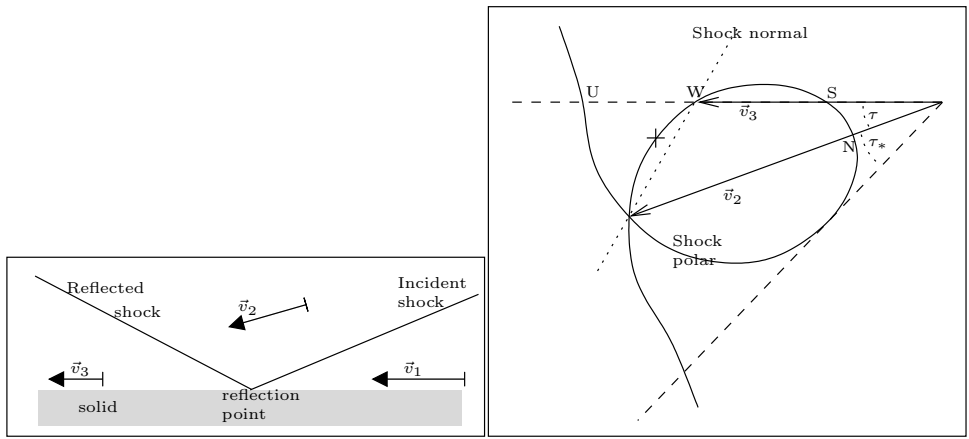

We temporarily drop the last requirement and consider the \defm{shock polar}: the curve of possible $\vv_3$
for various shock-wall angles. In the setting of Figure \myref{fig:spolar-right} right, 
there are exactly two\footnote{There is a third point which belongs to the
unphysical part of the shock polar, representing expansions shocks.
For other parameters (incident shocks), there may be \emph{no} reflected shock; 
in borderline cases there may be exactly one (called \defm{critical-type}). 
In such cases some type of Mach reflection should be expected.} points on the shock polar
that yield $\vv_3$ parallel to the wall. The corresponding shocks are called \defm{weak} ($W$) and \defm{strong} ($S$)
in the literature;
we prefer the terms \defm{weak-type} and \defm{strong-type}. 

While every steady shock must have a supersonic
upstream region, the downstream region may be subsonic (``transonic shock'') or supersonic (``supersonic shock'') or sonic. 
The weak-type shock can be any of these, but tends to be supersonic for the largest part of the parameter range.
The strong-type shock is always transonic.

Which of these two choices will occur? Chen and Feldman \cite{chen-feldman-selfsim-journal} constructed 
Figure \myref{fig:classrr} center left for $\theta\approx90^\circ$ with a \emph{weak-type supersonic} shock. 
\cite{elling-rrefl} obtained solutions for some $\theta\not\approx90^\circ$, but still with \emph{weak-type supersonic} shock. 

In other cases global strong-type reflections are known to exist. Consider the initial data of Figure \myref{fig:bigger90} left:
a straight shock separates two constant-state regions. If the parameters (wall angle, velocities, shock angle) are
chosen well, the strong-type reflected shock appearing for $t>0$ 
will be precisely perpendicular to the opposite wall (Figure \myref{fig:bigger90} center).

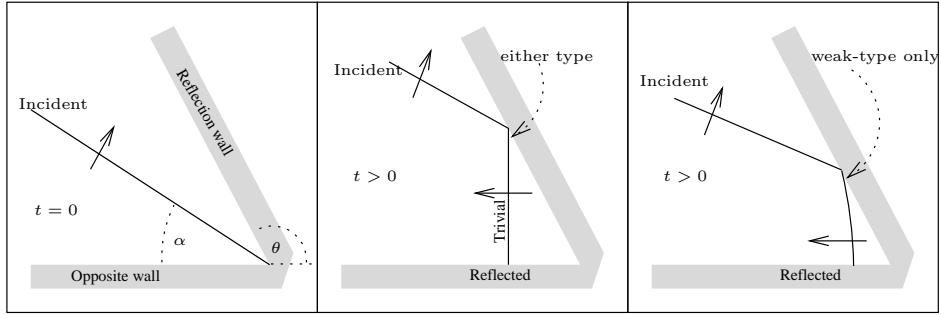
\begin{figure}[h]
\input{bigger90.pstex_t}%
\caption{Left: initial data. Center: a trivial case of global transonic RR. Right: a perturbation
that exists only for \emph{weak}-type RR.}%
\mylabel{fig:bigger90}%
\end{figure}

The local RR can be extended trivially into a global RR, with straight shocks separating constant states.
In particular, a global \emph{strong-type} RR of this kind is possible. 
However, \cite{elling-detachment} proves that this pattern is \defm{structurally unstable}:
when the parameters are perturbed, non-existence of a global strong-type RR can be shown in the class of
flows that have $C^1$ reflected shocks as well as continuous density and velocity in the triangular region enclosed by reflected shock
and wall corner.
\emph{Weak}-type transonic RR, on the other hand, is structurally stable in the same setting (see Figure \myref{fig:bigger90} right),
as \cite{elling-sonic-potf} has shown.

Naturally we wonder whether the previously mentioned problems, classical RR (Figure \myref{fig:classrr})
and supersonic wedge flow (Figure \myref{fig:superwedge}) allow some strong-type global RR for the \emph{same} parameters
that allow the already known solutions. This would constitute non-uniqueness examples for 2d compressible potential flow.

Uniqueness in \emph{general} function classes is far beyond state-of-the-art techniques.
Even uniqueness in $L^\infty$ or $\BV$ of a \emph{constant} state in 2d compressible potential\footnote{In fact we
are not aware of any proposals of admissibility criteria for multi-d compressible potential flow 
that apply to \emph{general} function classes
(as opposed to the Lax condition for piecewise smooth flow).}
flow appears to be open
(it is known for Euler flow, however \cite{dafermos-uq,diperna-uq}). 
This is a particular motivation for the present article: if existence of a second solution for the same
initial data could be shown, 
the initial-boundary value problem would be ill-posed. Indeed some researchers have suggested this is the case.
While for Euler flow, various rigorous or numerical non-uniqueness examples are known 
\cite{pullin-nuq,scheffer,shnirelman-1997,shnirelman-2000,elling-nuq-journal,de-lellis-szekelyhidi,lopes-lowengrub-lopes-zheng},
the author conjectures that it is caused by the presence of nonzero vorticity $\omega:=\nabla\times\vv$ 
and that uniqueness should hold for potential flow (Euler flow with the assumption of irrotationality, $\omega\equiv 0$),
if ``weak admissible solution'' is defined correctly.


Indeed, in the present article we prove that for potential flow
neither classical regular reflection (Figure \myref{fig:classrr} center right) nor supersonic wedge (Figure \myref{fig:superwedge} right)
have a strong-type \emph{global} RR solution, if we require (a) the reflected shocks to be $C^1$ and 
(b) the region between reflected shock and wall to have continuous velocity and density.

\section{Self-similar potential flow}

\subsection{Equations}

2d isentropic Euler flow is a PDE system for a density field $\rho$ and velocity field $\vv$,
consisting of the continuity equation
\begin{alignat}{1}
  \rho_t + \nabla\cdot(\rho\vv) &= 0 \mylabel{eq:continuity}
\end{alignat}
and the momentum equations
\begin{alignat}{1}
  (\rho\vv)_t+\nabla\cdot(\rho\vv\otimes\vv)+\nabla p &= 0
\notag\end{alignat}
The pressure $p$ is a strictly increasing smooth function of $\rho$. 
The \defm{sound speed} $c$ is 
$$c=\sqrt{\frac{dp}{d\rho}(\rho)}.$$
In this paper we focus on polytropic pressure:
\begin{alignat}{1}
    p(\rho) &= \rho^\gamma
\notag\end{alignat}
for some $\gamma\in[1,\infty)$. 

If we assume irrotationality
$$\nabla\times\vv,$$
then we may take
$$\vv=\nabla\phi$$
for a scalar potential $\phi$. Assuming smooth flow, the momentum equations yield
\begin{alignat}{1}
  \rho &= \pi^{-1}(A-\phi_t-\frac12|\nabla\phi|^2) \myeqlabel{eq:rhofun}
\end{alignat}
where $A$ is a global constant and where 
\begin{alignat}{1}
    \frac{d\pi}{d\rho}=\frac{1}{\rho}\cdot\frac{dp}{d\rho}=\rho^{-1}c^2. \myeqlabel{eq:pideriv}
\end{alignat}
The remaining continuity equation \myeqref{eq:continuity} is \defm{unsteady potential flow}.

For any $t\neq0$ we may change from standard coordinates $(t,x,y)$ to \defm{similarity coordinates} $(t,\xi,\eta)$
with $\vxi=(\xi,\eta)=(x/t,y/t)$. A flow is \defm{self-similar} if $\rho,\vv$ are functions of $\xi,\eta$ alone, without
explicit dependence on $t$. 

In potential flow, self-similarity corresponds to the ansatz
$$\phi(t,x,y)=t\psi(x/t,y/t).$$
By differentiating the divergence form \eqref{eq:continuity} of potential flow and using \myeqref{eq:rhofun} and \myeqref{eq:pideriv}, 
we obtain the non-divergence form
\begin{alignat}{1}
  (c^2I-(\nabla\psi-\vxi)^2):\nabla^2\psi &= 0. \myeqlabel{eq:nondivpsi}
\end{alignat}
Here $A:B$ is the Frobenius product $\trace(A^TB)$, $\vec w^2:=\vec w\otimes\vec w=\vec w\vec w^T$ (not $\vec w^T\vec w$) 
and $\nabla^2$ is accordingly the Hessian. In coordinates:
\begin{alignat*}{1}
  (c^2-(\psi_\xi-\xi)^2)\psi_{\xi\xi}-2(\psi_\xi-\xi)(\psi_\eta-\eta)\psi_{\xi\eta}
  +(c^2-(\psi_\eta-\eta)^2)\psi_{\eta\eta} &= 0.
\end{alignat*}
It is sometimes more convenient to use the \emph{pseudo-potential}
\begin{alignat}{1}&
    \chi:=\psi-\frac12|\vxi|^2
    \myeqlabel{eq:chipot}
\end{alignat}
which yields
\begin{alignat}{1}
  (c^2I-\nabla\chi^2):\nabla^2\chi + 2c^2 - |\nabla\chi|^2 = 0. \myeqlabel{eq:nondivchi}
\end{alignat}
We choose $A=0$ (by adding a constant to $\chi$) so that
\begin{alignat}{1}
    \rho &= \pi^{-1}\big(-\chi-\frac12|\nabla\chi|^2\big) \myeqlabel{eq:pi-ss}
\end{alignat}
Self-similar potential flow is a second-order PDE of \emph{mixed type}; 
the local type is determined by the coefficient matrix $c^2I-\nabla\chi^2$
which is positive definite if and only if $L<1$, where 
$$L:=\frac{|\vec z|}{c}=\frac{|\vv-\vec x/t|}{c}$$
is called \defm{pseudo-Mach number}; for $L>1$ the equation is hyperbolic.

\subsection{Symmetries}

Potential flow, like Euler and Navier-Stokes, has important symmetries that will simplify our discussion. 
First, it is invariant under rotation. Second, \myeqref{eq:nondivchi} is clearly translation-invariant:
if $\chi(\vxi)$ is a solution, so is $\chi(\vxi-\vec w)$. 
But in contrast to the steady flow, translation-invariance is not indifference of physics to the location
of an experiment; rather, it is the much less trivial invariance under change of \emph{inertial frame}.
In $(t,x,y)$ coordinates it corresponds to a change of observer
$$\vv\leftarrow\vv-\vec w,\qquad \vxi=\vec x/t \leftarrow\vxi-\vec w,$$
where $\vec w$ is the velocity of the new observer relative to the old one. Obviously the \defm{pseudo-velocity}
$$\vec z:=\nabla\chi=\nabla\psi-\vxi =\vv-\vxi$$
does not change.

\subsection{Slip condition}

At a solid wall we impose the usual \defm{slip condition}:
\begin{alignat}{1}
    0 &= \nabla\chi\cdot\vn = \nabla\psi\cdot\vn-\vxi\cdot\vn = \vv\cdot\vn-\vxi\cdot\vn = \vec z\cdot\vn.
    \myeqlabel{eq:slip}
\end{alignat}
In the frame of reference of an observer travelling on the wall, the slip condition takes the more familiar form
\begin{alignat}{1}
    0 &= \vv\cdot\vn,
    \myeqlabel{eq:slip0}
\end{alignat}
since the observer velocity $\vxi$ must satisfy $\vxi\cdot\vn=0$.

\subsection{Shock conditions}

The weak solutions of potential flow are defined by the divergence-form continuity equation \eqref{eq:continuity}.
Its self-similar form is
$$\nabla\cdot(\rho\nabla\chi)+2\rho = 0.$$
The corresponding Rankine-Hugoniot condition on a shock is 
\begin{alignat}{1}
  \rho_uz^n_u &= \rho_dz^n_d \myeqlabel{eq:rh-z} 
\end{alignat}
where $u,d$ indicate the limits on the \defm{upstream} and \defm{downstream} side and $z^n$, $z^t$ are the normal and tangential
component of $\vec z$.
As the equation is second-order, we must additionally require continuity of the potential:
\begin{alignat}{1}
  \psi^u &= \psi^d. \myeqlabel{eq:rh-cont}
\end{alignat}
By taking a tangential derivative, we obtain
\begin{alignat}{1}
  z^t_u &= z^t_d =: z^t \myeqlabel{eq:ztan}
\end{alignat}
It is easy to verify that translation- and rotation-invariance carry over to weak solutions.

Observing that $\sigma=\vxi\cdot\vn$ is the shock speed, we obtain the more familiar form
\begin{alignat}{1}
  \rho_uv^n_u - \rho_dv^n_d &= \sigma(\rho_u-\rho_d), \myeqlabel{eq:rh-v} \\
  v^t_u &= v^t_d =: v^t. \myeqlabel{eq:vtan}
\end{alignat}

Fix the unit shock normal $\vn$ so that $z^n_u>0$ (i.e.\ $\vn$ is pointing \defm{downstream}) which implies $z^n_d>0$ as well.
To avoid expansion shocks we must require the admissibility condition 
\begin{alignat}{1}&
    z^n_u\geq z^n_d, \myeqlabel{eq:adm-z}
\end{alignat}
which is equivalent to
\begin{alignat}{1}
  v^n_u &\geq v^n_d. \myeqlabel{eq:adm-v}
\end{alignat}
We chose the unit tangent $\vt$ to be $90^\circ$ counterclockwise from $\vn$.

By \myeqref{eq:vtan} the tangential components of the velocity are continuous across the shock,
so the velocity jump is normal. 
Assuming $v^n_u>v^n_d$ (positive shock strength), we can express the downstream shock normal as 
\begin{alignat}{1}
    \vn &= \frac{\vv_u-\vv_d}{|\vv_u-\vv_d|}. \myeqlabel{eq:normal-v}
\end{alignat}

If a shock meets a wall, with continuous $\rho,\vv$ on the $u,d$ sides near the meeting point, then
$\vec z_u$ and $\vec z_d$ must be tangential to the wall, by the slip condition \myref{eq:slip}.
Since $\vec z_u-\vec z_d$ is nonzero and normal to the shock, the shock must meet the wall at a \emph{right} angle.

\color{blue}

\subsection{Shock polar}

\color{black}

In our problem the upstream regions are constant and determined. Let $\psi$ be the potential in the downstream
region, $\psi^I$ the (linear) potential upstream (ditto for $\chi$, $\rho$, $z$, $v$, ...).
We substitute \myeqref{eq:rh-cont}, \myeqref{eq:pi-ss}, \myeqref{eq:normal-v} and \myeqref{eq:chipot}
into \myeqref{eq:rh-z} to obtain the shock condition 
\begin{alignat}{1}
    g(\nabla\psi,\vxi) &= 0
    \myeqlabel{eq:gshock}
\end{alignat}
where
\begin{alignat}{1}
  g(\vv,\vxi) 
  &:= 
  \Big(\pi^{-1}\big[ -\big( \psi^I(\vxi)
  -\frac12|\vxi|^2 \big)
  -\frac12|\vv-\vxi|^2\big]
  (\vv-\vxi)
  -\rho^I(\vv^I-\vxi)\Big)
  \cdot\frac{\vv^I-\vv}{|\vv^I-\vv|}
  \notag\\&=
  \Big(\pi^{-1}\big[ -\psi^I(\vxi) + \vv\cdot\vxi - \frac12|\vv|^2 \big]
  (\vv-\vxi)
  -\rho^I(\vv^I-\vxi)\Big)
  \cdot\frac{\vv^I-\vv}{|\vv^I-\vv|}
  . \myeqlabel{eq:g}
\end{alignat}
$g$ is smooth away from $\vv=\vv^I$ (which corresponds to a vanishing shock).

For any $\vxt$ such that $\vxt-\vxi\perp\vn=(\vI-\vv)/|\vI-\vv|$ 
we have $\psi^I(\vxt)=\psi^I(\vxi)+\vI\cdot(\vxt-\vxi)$, so 
\begin{alignat}{5}
    g(\vv,\vxt) 
    =&
    \Big(\pi^{-1}\big[ -\psi^I(\vxt) + \vv\cdot\vxt - \frac12|\vv|^2 \big]
    (\vv-\vxt)
    -\rho^I(\vv^I-\vxt)\Big)
    \cdot\vn
    \notag\\=& 
    \Big(\pi^{-1}\big[ -\psi^I(\vxi) - \vI\cdot(\vxt-\vxi) + \vv\cdot(\vxt-\vxi) + \vv\cdot\vxi - \frac12|\vv|^2 \big]
    (\vv-\vxi)
    -\rho^I(\vv^I-\vxi)\Big)
    \cdot\vn 
    \notag\\&+ \pi^{-1}\big[ -\psi^I(\vxt) + \vv\cdot\vxt - \frac12|\vv|^2 \big]\subeq{(\vxi-\vxt)\cdot\vn}{=0}
    - \rho^I\subeq{(\vxi-\vxt)\cdot\vn}{=0}
    \notag\\=& 
    \Big(\pi^{-1}\big[ -\psi^I(\vxi) + \subeq{\subeq{(\vv-\vI)}{\parallel\vn}\cdot(\vxt-\vxi)}{=0} + \vv\cdot\vxi - \frac12|\vv|^2 \big]
    (\vv-\vxi)
    -\rho^I(\vv^I-\vxi)\Big)
    \cdot\vn
    =
    g(\vv,\vxi) 
    \myeqlabel{eq:tan-shift}
\end{alignat}
as well. This corresponds to the well-known physical feature that if the downstream $\vv,\rho$ and shock tangent $T$ satisfy the shock relations 
in $\vxi$ (so that $g(\vv,\vxi)=0$), then also in $\vxt$ for any $\vxt$ on the tangent $T$ through $\vxi$.

Using \myeqref{eq:pideriv} and assuming $\psi$ locally satisfies the shock conditions \myeqref{eq:gshock} and \myeqref{eq:rh-cont}
(equivalently \myeqref{eq:rh-z} and \myeqref{eq:rh-cont}), we compute the derivatives:
\begin{alignat}{1}&
    \nabla_{\vv} \frac{\vv^I-\vv}{|\vv^I-\vv|} 
    = - |\vv^I-\vv|^{-1} \big(I-(\frac{\vv^I-\vv}{|\vv^I-\vv|})^2\big)  
    = - |\vv^I-\vv|^{-1} \big(I-\vn^2\big) 
    = - |\vv^I-\vv|^{-1} \vt^2
\notag\end{alignat}
(we remind that $\vec w^2=\vec w\vec w^T$) so that
\begin{alignat}{1}
    g_{\vv} := (\frac{\partial g}{\partial v^x},\frac{\partial g}{\partial v^y})
    &= 
    \rho\Big(I-(\frac{\vv-\vxi}{c})^2\Big)\vn
    -
    \frac{\rho(\vv-\vxi)-\rho^I(\vv^I-\vxi)}{|\vv^I-\vv|}\cdot\vt^2. \myeqlabel{eq:gv}
\end{alignat}
Therefore
\begin{alignat}{1}
    g_{\vv}\cdot\vn 
    &=
    \rho\big(1-(\frac{z_n}{c})^2\big)
    > 0
    \myeqlabel{eq:gvn}
\end{alignat}
for admissible shocks (with nonzero strength),
so $g_{\vv}\neq 0$ always.
On the other hand
\begin{alignat}{1}
    g_{\vv}\cdot\vt 
    &\topref{eq:ztan}=
    \rho(-\frac{z^tz^n}{c^2})
    - \frac{\rho-\rho^I}{|\vv^I-\vv|}z^t
    = - z^t \big( \subeq{\rho\frac{z^n}{c^2}}{>0} 
    + \subeq{(\rho-\rho^I)}{>0}|\vv^I-\vv|^{-1} \big)
    \myeqlabel{eq:gvt}
\end{alignat}
so that
\begin{alignat}{1}
    \sign(g_{\vv}\cdot\vt) &= - \sign z^t.
    \myeqlabel{eq:gvtsign}
\end{alignat}

The shock polar (see Figure \myref{fig:spolar-right}) is the curve of $\vv$ obtained
by holding the shock in a fixed $\vxi$ 
and keeping the upstream state fixed while varying the normal. 
Therefore the shock polar is the curve of solutions $\vv$ of 
$$g(\vv,\vxi)=0.$$
Hence 
$$g_{\vv}(\vv,\vxi) \perp \text{shock polar}\quad\text{in $\vv$,}$$ 
by the implicit function theorem.

In Figure \myref{fig:spolar-right} right the point $N$ of the polar corresponds to a \emph{pseudo-normal} 
shock: $z^t=0$. In $N$, the normal $g_{\vv}$ points (by \myeqref{eq:gvn}) in the same direction as 
$$\vn=\frac{\vv_2-\vv_3}{|\vv_2-\vv_3|},$$
hence (for a $\vv_3$ ending in $N$ and $\vv_2$ as shown) left.
Therefore $g_{\vv}$ is an \emph{inner} normal \footnote{not necessarily unit} 
to the admissible part of the shock polar. 

In local RR the reflected shock must yield $\vv_3$ parallel to the wall. 
In Figure \myref{fig:spolar-right} right, $\vv$ for the weak-type shock 
(base in origin, tip in W) yields $\vv\cdot\vn<0$ for inner normals $\vn$ of the shock polar 
whereas $\vv$ for the strong-type shock (tip in $K$) yields $\vv\cdot\vn >0$.
A critical-type shock (see $\tau_*$ in Figure \myref{fig:spolar-right} right)
is the limit of adjacent weak and strong types, so $\vv\cdot\vn=0$. 
This motivates the following definition:
\begin{definition}
    \mylabel{def:type}%
  A shock is called \defm{weak-type} (in a particular point $\vxi$ in self-similar coordinates) if 
  \begin{alignat}{1}
    g_{\vv}\cdot\vz &< 0, \myeqlabel{eq:weaktype}
  \end{alignat}
  (where $\vz$ is still downstream), \defm{strong-type} if $>0$, \defm{critical-type} if $=0$.
\end{definition}
The definition has three pleasant properties: it coincides with the standard definition in the case 
of strictly convex polars,
it generalizes the definition of weak/strong-type to non-convex cases\footnote{In such cases, there
may be three or more reflected shocks that yield $\vv_3$ tangential to the wall.}, and finally 
the sign condition is precisely what is needed for discussing elliptic corner regularity (see \cite{elling-sonic-potf}). 

\cite[Theorem 1]{elling-sonic-potf} asserts that the (physical part of the) shock polar is strictly convex
for potential flow with polytropic pressure law, the case we consider here.

\color{black}

\section{Considerations for transonic reflected shocks of either type}

In this section we allow the reflected shock (transonic) to be any type.
We show that after a change of coordinates the minimum of $\psi$ over the elliptic region is attained in the reflection point. 
In the next section we focus on a strong-type reflected shock and obtain a contradiction by ruling out a minimum in the reflection point.

\subsection{Classical regular reflection}

Consider the possibility of a transonic (as in Figure \myref{fig:classrr} center right) global solution 
of classical regular reflection. 

Using invariance under rotation and change of observer, we may assume coordinates have been chosen 
(see Figure \myref{fig:clarr} left and right) 
so that the constant velocity $\vv^I$ on the hyperbolic side of the reflected shock is vertical down and so that
$\vv$ approaches $0$ as we approach the reflection corner through the elliptic region $E$.
Both combined, $\vv^I-\vv=\vv^I$ --- which is the shock normal, by \myeqref{eq:normal-v} --- is vertical down,
so the tangent of the reflected shock is horizontal in the reflection point.

Let $S$ be the reflected shock, $A$ the reflection wall, $B$ the opposite wall, 
$B_I$ and $A_I$ the parts above the shock while $B_E$ and $A_E$ are the segments below the shock;
all these sets are meant to \emph{exclude} endpoints.
$E$ is the elliptic region, $I$ the hyperbolic region adjacent to $B$; $E,I$ are meant to be \emph{open}. 
Let $\vxi_r$ be the reflection point.
The unit normals $\vn_B$ of $B$ and $\vn_A$ of $A$ are chosen \emph{outer} to $E$.

We assume\footnote{The notion of weak- and strong-type
loses meaning if we do not require $\vv=\nabla\psi$ to have the same limit on the shock and at the wall as we approach
the reflection corner. The question studied in this paper makes no sense if we require less than $C^1$ regularity in the corner.} 
$\psi\in C^1(\overline E)$ so that $\rho\vv\in C^0(\overline E)$; we also assume $\overline S$ (including endpoints) is $C^1$. 
($\psi$ is affine in the hyperbolic regions, yielding constant $\rho,\vv$.)
From now on, $\psi$ is always meant to be the restriction of $\psi$ to $\overline E$, with limits on $\partial E$ taken in $E$.
In particular, ``global'' extremum refers to the extremum over $\overline E$.

\if\dofigures%
\begin{figure}
\input{clarr.pstex_t}
\caption{Two cases of wall-corner-induced RR.}
\mylabel{fig:clarr}
\end{figure}
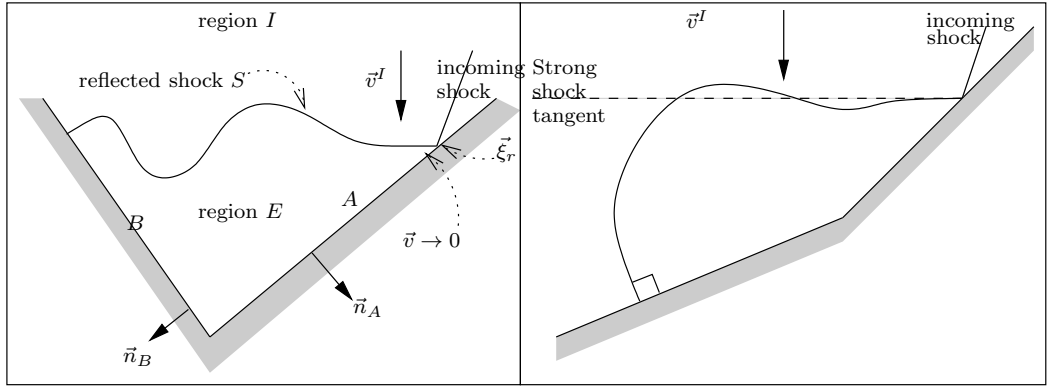
\fi%

\paragraph{Cases}

Consider the angle $\measuredangle(\vv^I,\vn_B)$ between $\vv^I$ and $\vn_B$. There are three cases: $>90^\circ$, $=90^\circ$, $<90^\circ$ (the latter includes in particular all cases of classical RR).

For $=90^\circ$, strong-type global reflection exists, as observed in the introduction.
But otherwise we can prove non-existence. Consider the $<90^\circ$ case. 

If $\psi$ was affine ($\rho,\vv$ constant), then the shock $S$ would be straight and horizontal.
But then it would meet the opposite wall $B$ at an angle $\neq 90^\circ$ or not at all --- contradiction.
Therefore $\psi$ is in particular not constant.

\paragraph{Extrema}

$\psi$ is continuous, in particular, so it must attain a global minimum in $\overline E$ which is compact. 
Assume $\psi$ does \emph{not} attain its minimum in the reflection point $\vxi_r$. 

\paragraph{Opposite wall}

On $B_I$ the slip condition \myeqref{eq:slip} implies
$$ 0 = \subeq{\vv^I\cdot\vn_B}{>0} - \vxi\cdot\vn_B \quad\Rightarrow\quad\vxi\cdot\vn_B>0.$$
$\vxi\cdot\vn_B>0$ is constant along $B$, so at $B_E$ we also have
$$ 0 \topref{eq:slip}{=} \nabla\chi\cdot\vn_B = \nabla\psi\cdot\vn_B-\subeq{\vxi\cdot\vn_B}{>0}\quad\Rightarrow\quad\nabla\psi\cdot\vn_B>0. $$
This rules out a local minimum at $\overline{B_E}$, including wall-wall and wall-shock corner. 
(In this step we see the key difference to the case $\measuredangle(\vv^I,\vn_B)=90^\circ$.)

\paragraph{Interior}

By the strong maximum principle, \myeqref{eq:nondivpsi} does not allow local $\psi$ extrema in the interior $E$,
since we have already shown $\psi$ is not constant.

\paragraph{Reflection wall}

By choice of coordinates, 
\begin{alignat}{1}
    \lim_{E\ni\vxi\rightarrow\vxr} \vv(\vxi) &= 0 \myeqlabel{eq:psixr}
\end{alignat}
The slip condition at $\vxr$ is
$$ 0 = \nabla\chi\cdot\vn_A = \nabla\psi\cdot\vn_A -\vxr\cdot\vn_A = \vv\cdot\vn_A-\vxr\cdot\vn_A 
\topref{eq:psixr}{=} -\vxr\cdot\vn_A\quad .$$
This implies that for every other $\vxi\in A $
$$ 0 = \vxi\cdot \vn_A  $$
as well. Therefore the slip condition yields
$$ 0 = \nabla\psi\cdot\vn_A - \vxi\cdot\vn_A = \nabla\psi\cdot\vn_A \qquad\text{on $A_E$.} $$
Combined with \myeqref{eq:nondivpsi} the Hopf lemma \cite[Lemma 3.4]{gilbarg-trudinger} rules out a local minimum of $\psi$ at $A_E$.

\paragraph{Shock}

Hence the global minimum of $\psi$ can only be attained in a point $\vxs \in S $ at the shock, 
away from both endpoints, and since it is not attained in the reflection corner $\vxr$, 
necessarily 
$$\psi(\vxs)<\psi(\vxr).$$
Combined with the shock condition 
$$ \psi \topref{eq:rh-cont}{=} \psi^I = \psi^I(0) + \vv^I\cdot\vxi $$
this implies $\eta_s > \eta_r$ since $\vv^I$ is vertical down (Figure \myref{fig:clarr}).
A $\psi$ minimum requires 
\begin{alignat}{1}&
    \nabla\psi\cdot\vt = 0 \qquad\text{in $\vxs$,}
\notag\end{alignat}
so that the shock is horizontal, as well as 
\begin{alignat}{1}&
    \nabla\psi\cdot\vn \geq 0 \qquad\text{in $\vxs$,}
\notag\end{alignat}
where $\vn$ is the \emph{downstream} normal, hence \emph{inner} to $E$.
Since the minimum is global, $\vxs$ is the highest point of $\overline E$, so $\vn$ points vertically down:
\begin{alignat}{1}&
    \psi_\eta(\vxs) \leq 0. \myeqlabel{eq:psie1}
\end{alignat}
The reflected shock $S$ is normal ($\vv^I\parallel\vn$) both at $\vxr$ and in $\vxs$,
but higher in $\eta_s$. Being farther upstream in $\vxi$ coordinates corresponds 
to moving faster upstream in $(t,\vec x)$ coordinates. 
A normal shock is the stronger the faster it moves upstream. 
With upstream velocity held fixed, that means the downstream velocity $\vv_d\cdot\vn$ becomes smaller
($\vn$ pointing downstream). In our context that means $\psi_\eta$ increases. 
(This argument is contained in \cite[Proposition 2.9]{elling-liu-pmeyer} whose proof provides a detailed calculation.)
Since $v^y=\psi_\eta=0$ at $\vxr$ on the $E$ side, necessarily 
$$\psi_\eta(\vxs) > 0, $$ 
in contradiction to \myeqref{eq:psie1}.

\paragraph{Conclusion}

We have ruled out a global minimum in every point of $\overline E$ other than $\vxr$, so our original assumption was wrong: 
\begin{proposition}
    For $\measuredangle(\vv^I,\vn)\neq 90^\circ$,
    $\psi$ is not constant and must attain\footnote{This is also true for $=90^\circ$, but irrelevant, and would require some proof modifications.} 
    its global minimum in $\vxr$. 
\end{proposition}
The arguments above apply to $<90^\circ$; 
the case $>90^\circ$ is analogous, using global maxima instead of minima, with obvious modifications.

\subsection{Supersonic wedge flow}

\if\dofigures%
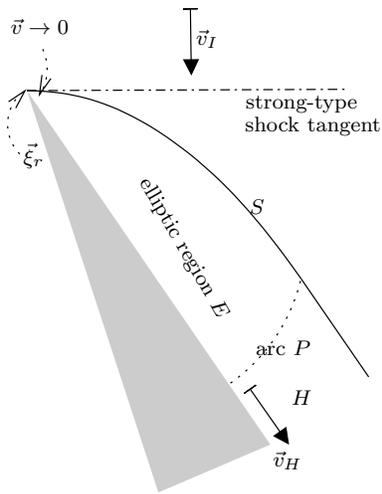
\begin{figure}
    \centerline{\input{super2.pstex_t}}
    \caption{Transonic reflected shock in supersonic wedge flow.}
    \mylabel{fig:super2}
\end{figure}
\fi%

The arguments for the supersonic wedge are similar. We consider (see Figure \myref{fig:super2})
a transonic reflected shock at the wedge tip $\vxr$,
with a hyperbolic upstream region with constant velocity $\vv^I$, coordinates shifted and rotated so that $\vv^I$ is vertical down 
and so that $\vv$ converges to $0$ as we approach the reflection point on the downstream side of the reflected shock.

The shock $\overline S$ is assumed to be $C^1$ (including the endpoint $\vxr$). 
Below it and adjacent to $\vxr$ is the elliptic region $E$.
$E$ is bounded by a circular arc $P$ (where \myeqref{eq:nondivpsi} becomes parabolic, with $L=1$), 
with an infinite hyperbolic region $H$ with constant velocity $\vv_H$ on the other side. 
We assume $\psi\in C^1(\overline{E\cup H})$. The portion of $S$ adjacent to $H$ is 
straight and parallel to the wedge boundary (else the initial data would be different from the case of \cite{elling-liu-pmeyer}). 

\begin{proposition}
    $\psi$ is not constant in $E$ and attains its global minimum in $\vxr$.
\end{proposition}
\begin{proof}
    If $\psi$ was affine ($(\rho,\vv)$ constant) in $E$, then the shock would be straight, but it has to become parallel
    to the downstream wall near infinity --- contradiction to $C^1$. So $\psi$ is in particular not constant.

    $\psi$ must attain a global minimum over the compact region $\overline E$. 
    Suppose it does not attain it in $\vxr$.
    
    A global minimum of $\psi$ in the interior or at the wall or shock (excluding endpoints) is 
    ruled out in the same manner as above for classical reflection (with the wall in the same role as the reflection wall $A_E$).
    Also as before we note that due to \myeqref{eq:psixr}
    we have $\vxi\cdot\vn=0$ on the wall so that
    the slip condition \myeqref{eq:slip} takes the form
    $$ 0 = \vv\cdot \vn. $$
    Therefore, the (constant) velocity $\vv_H$ in $H$ must be parallel to the wall. 
    
    At $H$, the (constant) normal of $S$ points down and left, so since the upstream velocity $\vv^I$ is vertical down 
    and --- for an admissible shock --- larger than the downstream velocity $\vv_H$, \myeqref{eq:normal-v} implies
    $\vv_H$ points down and right. Therefore 
    $$\nabla\psi\cdot\vn_{E\rightarrow H}>0 \qquad\text{on $\overline{P}$,}$$
    where $\vn_{E\rightarrow H}$ is the unit normal of $P$ \emph{outer} to $E$. Hence $\psi$ cannot attain a local minimum
    at $P$. All minimum locations in $E$ other than $\vxr$ have been ruled out.
\end{proof}

\section{Non-existence for strong-type shocks}

We have shown that minima can only be attained in the reflection point. 
Now we assume, in addition, that the reflected shock is strong-type and obtain a contradiction:
the minimum cannot be attained in the reflection corner either. 

Suppose $\psi$ does have a strict local minimum in $\vxr$ (again the case of maxima is analogous).
We will obtain a contradiction by constructing a subsolution $\suso$.

%

\paragraph{Uniform coordinates}

For convenience we may rotate and mirror-reflect,
to bring the reflection corner of either problem into the coordinates of Figure \myref{fig:corner},
with the wall $A$ emanating into positive horizontal direction from $\vxr$ and $T$ emanating into the first quadrant.
(Since we \emph{preserve} the origin of similarity coordinates $\vxi$, this does not change the values of $\psi$
which still attains a local minimum in $\vxr$.)
\if\dofigures%
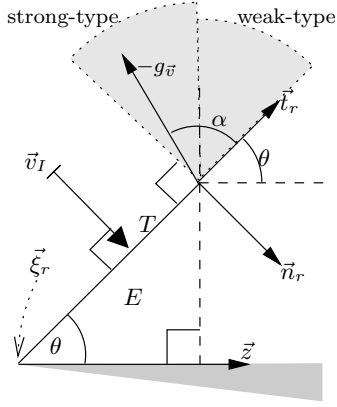
\begin{figure}
\centerline{\input{corner.pstex_t}}
\caption{Near the reflection point.}
\mylabel{fig:corner}
\end{figure}
\fi%

\color{blue}

\paragraph{Boundary conditions}

\color{black}

On the shock $S$ we use the shock condition \myeqref{eq:gshock}.
Let $T$ be the shock tangent in the reflection point $\vxr$,
$\vn_r$ the downstream normal of $T$ and $\vt_r$ the corresponding tangent (counterclockwise from
$\vn_r$, by convention). 
Set
\begin{alignat}{1}
    \alpha &:= \measuredangle\big(\vt_r,-g_{\vv}(0,\vxr)\big), \myeqlabel{eq:alphadef}
\end{alignat}
where $\measuredangle(\vec a,\vec b)$ is the counterclockwise angle from $\vec a$ to $\vec b$.
By Definition \myref{def:type}, the shock is strong-type in $\vxr$ if and only if 
$$-g_{\vv}(0,\vxr)\cdot\vec z<0.$$
$\vz(\vxr)$ is tangential to the wall, 
hence horizontal, and pointing downstream, hence right:
\begin{alignat}{1}&
    z^x(\vxr) > 0, \qquad z^y(\vxr) = 0. \myeqlabel{eq:zdpoint}
\end{alignat}
Therefore $\vz_r\cdot\vt_r>0$ so that
\begin{alignat}{1}&
    -g_{\vv}(0,\vxr)\cdot\vt_r \topref{eq:gvtsign}{>} 0.
    \myeqlabel{eq:gvtsignr}
\end{alignat}
Moreover 
\begin{alignat}{1}&
    -g_{\vv}(0,\vxr)\cdot\vn_r \topref{eq:gvn}{<} 0.
    \myeqlabel{eq:gvpos}
\end{alignat}
Both combined (see Figure \myref{fig:corner}):
\begin{alignat}{1}
    \alpha+\theta &\begin{cases} 
        \in(\theta,90^\circ), & \text{shock weak-type}, \\
        =90^\circ, & \text{shock critical-type}, \\
        \in(90^\circ,90^\circ+\theta), & \text{shock strong-type}.
    \end{cases}.
    \mylabel{eq:gamma}
\end{alignat}

\color{black}

\paragraph{Dilation}

$\vv(\vxr)=0$ by choice \myeqref{eq:psixr}, so 
$$ \vxr = \vz(\vxr)-\vv(\vxr) = \vz(\vxr) \topref{eq:slip}{\parallel} \text{wall (horizontal)} ,$$
and moreover in $\vxr$ the PDE \myeqref{eq:nondivpsi} has the form
\begin{alignat}{1}
    (I-c^{-2}\vxr\vxr^T):\nabla^2\psi &\topref{eq:nondivpsi}{=} 0;   \myeqlabel{eq:psicorner}
\end{alignat}
we may change coordinates by dilating in the horizontal direction to transform the PDE to
\begin{alignat}{1} 
    \Delta\psi &= 0.
\notag\end{alignat}
The wall boundary condition remains $\psi_\eta=0$, and while $g_{\vv}(0,\vxr)$, $\alpha$, $\theta$
may change to some $\tilde g_{\vv}$, $\tilde\alpha$, $\tilde\theta$, the property
\begin{alignat}{1}&
    90^\circ < \tilde\alpha+\tilde\theta < 180^\circ \myeqlabel{eq:tilde-at}
\end{alignat}
(compare \myeqref{eq:gamma}) is \emph{preserved} by the dilation (see Figure \myref{fig:corner}).

\paragraph{Subsolution}

Now change to polar coordinates $(r,\phi)$ centered in the reflection point $\vxr$.
We let $\phi=0^\circ$ represent the wall while $\phi=\theta$ represents [the image under dilation of] $T$.

We seek $\suso$ in the form
\begin{alignat}{1}
    \suso(r,\phi)
    &=
    \psi^I(\vxr) + \epsilon r \cos(\beta\phi) 
    \myeqlabel{eq:suso}
\end{alignat}
where $\epsilon\in(0,1)$ will be small while $\beta\in(0,1)$ will be taken close to $1$. 
\begin{alignat}{1}
    \suso_r &= \epsilon\cos(\beta\phi),\qquad r^{-1}\suso_\phi = - \epsilon\beta \sin(\beta\phi),
\notag\end{alignat}
so
\begin{alignat}{1}
    |\nabla\suso| &= O(\epsilon) \qquad\text{as $r\downarrow 0$;}
    \myeqlabel{eq:Dsuso}
\end{alignat}
moreover on the wall $\phi=0^\circ$ the slip condition 
\begin{alignat}{1}
    0 &= \nabla\suso\cdot\vn=-r^{-1}\suso_\phi = \epsilon\beta \sin(\beta\phi) 
    \myeqlabel{eq:slipPsi}
\end{alignat}
is already satisfied.
\begin{alignat}{1}
    |\nabla^2\suso| &= O(\epsilon r^{-1}) \qquad\text{as $r\downarrow 0$.}
\notag\end{alignat}
In the interior near $\vxr$, 
\begin{alignat}{1}
     -\Delta\suso 
     &=
     -\suso_{rr}-r^{-1}\suso_r-r^{-2}\suso_{\phi\phi} 
     = \epsilon r^{-1}(\subeq{\beta^2}{<1}-1)\subeq{\cos(\beta\phi)}{>0}
     \leq -\delta\epsilon r^{-1}
     \myeqlabel{eq:intdel}
\end{alignat}
for some $\delta=\delta(\beta)>0$ independent of $\epsilon,\vxi$,
since $\beta\in(0,1)$ and $\theta\in(0^\circ,90^\circ)$ imply $\beta\phi\in
\subset(0^\circ,90^\circ)$.
On the reflection point shock tangent $T$,
\begin{alignat}{1}
    -\frac{\tilde g_{\vv}}{|\tilde g_{\vv}|}\cdot\nabla\suso 
    &= \suso_r\cos\tilde\alpha + r^{-1}\suso_\phi\sin\tilde\alpha 
    \notag\\&= \epsilon\big(\cos(\beta\tilde\theta)\cos\tilde\alpha-\beta\sin(\beta\tilde\theta)\sin\tilde\alpha\big)
    \notag\\& = \epsilon\big(\subeq{(1-\beta)}{\approx 0}\cos\tilde\alpha\cos(\beta\tilde\theta)
    +\beta\subeq{\cos(\tilde\alpha+\beta\tilde\theta)}{<0}\big)
    \leq -\epsilon\delta 
\notag\end{alignat}
for some $\delta>0$ independent of $\epsilon,\vxi$ if we choose $\beta<1$ sufficiently close to $1$, because 
\myeqref{eq:tilde-at} implies 
\begin{alignat}{1}&
    90^\circ < \tilde\alpha + \beta\tilde\theta < 180^\circ \qquad\text{for $\beta\approx 1$}
\notag\end{alignat}
We obtain
\begin{alignat}{1}
    -\tilde g_{\vv}\cdot\nabla\suso 
    &\leq -\epsilon\delta 
    \myeqlabel{eq:shockg}
\end{alignat}
for a modified $\delta>0$. Finally, note that
\begin{alignat*}{5}
    \psi_r(0,0) - \suso_r(0,0) &\toprefb{eq:psixr}{eq:suso}{=} 0 - \epsilon\cos(\beta\cdot 0) = -\epsilon < 0 
\end{alignat*}
so that $r\mapsto\psi(r,0)-\suso(r,0)$ is strictly decreasing in $r=0$; hence
\begin{alignat}{5}
    \text{$\psi-\suso$ does \emph{not} attain a local minimum in $\vxr$.}
    \myeqlabel{eq:psi-suso-min}
\end{alignat}

\color{blue}

\paragraph{Undilated coordinates}

\color{black}

We return to undilated coordinates.
Change the definition of $r$ to the comparable $r:=|\vxi-\vxr|$ from now on.
Then 
\begin{alignat}{1}
    -g_{\vv}(0,\vxr)\cdot\nabla\suso(\vxi)
    &\topref{eq:shockg}{\leq}
    -\delta\epsilon
    \myeqlabel{eq:gvsuso}
\end{alignat}
and 
\begin{alignat}{1}
    -[I-c(\vxr)^{-2}\nabla\chi(\vxr)^2]:\nabla^2\suso(\vxi)
    &\topref{eq:intdel}{\leq} -\delta\epsilon r^{-1}
    \myeqlabel{eq:Asuso}
\end{alignat}
for some other $\delta>0$ independent of $\epsilon,\vxi$. 
\begin{alignat}{1}
    0 = \nabla\suso\cdot\vec n \qquad\text{on the wall}
    \myeqlabel{eq:unPsi}
\end{alignat}
is unchanged (from \myeqref{eq:slipPsi}) since the dilation was in the wall direction.

We focus on a small ball $B_R(\vxr)$ with radius $R>0$ centered in $\vxr$. 
Since $\psi-\suso$ is continuous on the compact set $\overline E\cap\overline B_R(\vxr)$, it must attain
a minimum on that set. We have already excluded a local minimum in $\vxr$ in
\myeqref{eq:psi-suso-min}. On the other hand $\psi(\vxr)-\suso(\vxr)=0$, so that the actual minimum must be \emph{negative}. 
The rest of the compact set is covered by $\partial B_R(\vxr)\cap\overline E$ (arc), $S\cap B_R(\vxr)$ (shock), $A\cap B_R(\vxr)$ (wall), 
and $E\cap B_R(\vxr)$ (interior). 

%

\color{black}

\paragraph{Interior} Using $\psi\in C^1(\overline E)$ and therefore $\nabla\chi,c\in C^0(\overline E)$,
we have for $\vxi\in B_R(\vxr)\cap E$ that
\begin{alignat}{1}
    -[I-c(\vxi)^{-2}\nabla\chi(\vxi)^2]:\nabla^2\suso(\vxi)
    &=
    -[I-c(\vxr)^{-2}\nabla\chi(\vxr)^2+o_r(1)]:\subeq{\nabla^2\suso}{=O(\epsilon r^{-1})}
    \notag\\
    &\topref{eq:Asuso}{\leq}
    (-\delta r^{-1} + o(r^{-1}))\epsilon 
    <
    0.
    \myeqlabel{eq:intvar}
\end{alignat}
if $R>r>0$ is sufficiently small. 
Hence 
\begin{alignat}{1}
    -[I-c^{-2}\nabla\chi^2]:\nabla^2(\psi-\suso) &\topref{eq:nondivpsi}{>} 0
    \qquad\text{on $B_R(\vxr)\cap E$}
    \myeqlabel{eq:diffint}
\end{alignat}
so that $\psi-\suso$ cannot have a local minimum in $E\cap B_R(\vxr)$, by the weak maximum principle.

\paragraph{At the wall} The Hopf lemma \cite[Lemma 3.4]{gilbarg-trudinger}, using \myeqref{eq:diffint} and $\nabla(\psi-\suso)\cdot\vn=0$ 
(by \myeqref{eq:unPsi}), rules out a local minimum at the wall.

\color{blue}

\paragraph{At the shock} 

\color{black}


Now assume $\psi-\suso$ has a \emph{negative} minimum 
in $\vxs\in S\cap B_R(\vxr)$. Then 
\begin{alignat}{1}
    & \psi-\suso < 0 \qquad\text{and} \myeqlabel{eq:diff} \\
    & \nabla(\psi-\suso)\cdot\vt = 0 \qquad\text{and} \myeqlabel{eq:difft} \\
    & \nabla(\psi-\suso)\cdot\vn \geq 0 \qquad\text{in $\vxs$} \myeqlabel{eq:diffn} 
\end{alignat}
(again $\vn$ is the \emph{downstream} normal, hence \emph{inner}). 

Let $\vxt$ be the closest point on $T$ to $\vxs$. Then
\begin{alignat*}{5}
    \psi(\vxs)
    \topref{eq:rh-cont}{=}
    \psi^I(\vxs)
    =
    \psi^I(\vxt) + \vI\cdot(\vxs-\vxt)
    =
    \psi^I(\vxr) + \subeq{\vI\cdot(\vxt-\vxr)}{=0} + \subeq{\vI}{\neq 0}\cdot(\vxs-\vxt)
\end{alignat*}
since $\vI\perp T\parallel\vxt-\vxr$ and $\vxr,\vxt\in T$. 
On the other hand
\begin{alignat*}{5}
    \psi(\vxs) 
    \topref{eq:diff}{<}
    \suso(\vxs)
    \topref{eq:suso}{=}
    \psi^I(\vxr) + O(\epsilon R)
\end{alignat*}
Combining both and using that $\vxt-\vxi\parallel\vI$ we obtain 
\begin{alignat}{5}
    |\vxt-\vxs| &= 
    O(\epsilon R)
    \label{eq:vxt-vxs-distance}
\end{alignat}
Moreover \myeqref{eq:difft} and \myeqref{eq:diffn} yield
\begin{alignat}{5}
    \nabla\psi(\vxs) &= a\vn(\vxs) + \nabla\suso(\vxs) \quad\text{where}\quad a\geq 0 \quad.
    \myeqlabel{eq:c}
\end{alignat}
\myeqref{eq:gvn} yields
\begin{alignat}{5}
    \gv(\nabla\psi(\vxs),\vxs) \cdot \vn(\vxs) &\geq C > 0 \myeqlabel{eq:CC}
\end{alignat}
where $C$ does not depend on $R,\epsilon$. 
\begin{alignat*}{5}
    \nabla\psi(\vxs) &= \subeq{\nabla\psi(\vxr)}{\topref{eq:psixr}{=}0} + o_r(1) \quad, 
\end{alignat*}
so we obtain for $t\in[0,1]$ and $\vxt$ as above that
\begin{alignat}{5}
    \gv(t\nabla\psi(\vxs),\vxt) \cdot \vn(\vxs)
    &=
    \big[\gv(\nabla\psi(\vxs),\vxs) + (1-t)O(|\nabla\psi(\vxs)|) + O(|\vxs-\vxt|)\big] \cdot \vn(\vxs)
    \notag\\&=
    \big[\gv(\nabla\psi(\vxs),\vxs) + o_R(1) + O(\epsilon R)\big] \cdot \vn(\vxs)
    \topref{eq:CC}{\geq} C + o_R(1)
    \geq 0
    \myeqlabel{eq:gv-ns}
\end{alignat}
for $R>0$ sufficiently small, \emph{not} depending on $\epsilon$. 
Moreover, \myeqref{eq:gvsuso} yields 
\begin{alignat}{5}
    \gv(t\nabla\psi(\vxs),\vxt) \cdot \nabla\suso(\vxs)
    &=
    \big[ \gv(0,\vxr) + \subeq{O(|t\nabla\psi(\vxs)|)}{=o_R(1)} + \subeq{O(|\vxt-\vxr|)}{=O(R)} \big] \cdot \subeq{\nabla\suso(\vxs)}{\topref{eq:Dsuso}{=}O(\epsilon)}
    \notag\\&
    \topref{eq:gvsuso}{\geq}
    (\delta + o_R(1)) \epsilon
    \geq
    \frac{\delta}{2} \epsilon
    \myeqlabel{eq:gv-dsuso}
\end{alignat}
if $R$ is sufficiently small, again \emph{not} depending on $\epsilon$. Finally, 
\begin{alignat*}{5}
    0 &= g(\nabla\psi(\vxs),\vxs) 
    \topref{eq:vxt-vxs-distance}{=} g(\nabla\psi(\vxs),\vxt) + O(\epsilon R)
    \\&= \subeq{ g(0,\vxt) }{ \topref{eq:tan-shift}{=} g(0,\vxr) = 0 } + \int_0^1 g_{\vv}(t\nabla\psi(\vxs),\vxt) dt \cdot \nabla\psi(\vxs) + O(\epsilon R)
    \\&\topref{eq:c}{=} \subeq{a}{\geq 0} \int_0^1 \subeq{ g_{\vv}(t\nabla\psi(\vxs),\vxt) \cdot \vn(\vxs)}{\topref{eq:gv-ns}{\geq} 0} dt 
    + \int_0^1 \subeq{g_{\vv}(t\nabla\psi(\vxs),\vxt) \cdot \nabla\suso(\vxs)}{\topref{eq:gv-dsuso}{\geq}\frac12\delta\epsilon} dt
    + O(\epsilon R)
    \\&\geq 
    \big(\frac12 \delta+O(R)\big)\epsilon 
    > 0
\end{alignat*}
if $R$ is sufficiently small, depending on $\delta$ but \emph{not} on $\epsilon$. 

We have a contradiction. Hence $\psi-\suso$ cannot have a negative minimum at $\vxs\in S\cap B_r(\vxr)$.

\paragraph{Arc}

$\psi$ has a strict local minimum in $\vxr$, so $\psi(\vxi) > \psi(\vxr) = \psi^I(\vxr)$ for every $\vxi\in\partial B_R(\vxr)\cap\overline E$.
By continuity of $\psi$ and compactness of $\partial B_R(\vxr)\cap\overline E$, 
\begin{alignat*}{5}
    \min_{\partial B_R(\vxr)\cap\overline E} (\psi-\psi^I(\vxr))
    &> 
    0
\end{alignat*}
On the other hand
\begin{alignat*}{5}
    \max_{\partial B_R(\vxr)\cap\overline E} (\suso-\psi^I(\vxr))
    \topref{eq:suso}{=}
    O(\epsilon)
\end{alignat*}
so that
\begin{alignat*}{5}
    \min_{\partial B_R(\vxr)\cap\overline E} (\psi-\suso) &> 0 
\end{alignat*}
if $\epsilon$ is sufficiently small. 

Since $\psi-\suso=0$ in $\vxr$, this means $\psi-\suso$ does not attain its minimum with respect to $\overline B_R(\vxr)\cap\overline E$
on $\partial B_R(\vxr)\cap\overline E$. 
Since all possible locations for minima have been ruled out, we obtain a contradiction.

Hence our assumption was wrong; $\psi$ cannot attain its $\overline E$ minimum in $\vxr$ either. 

\color{black}

\section{Conclusion}

Combining the results from the previous two sections, we have shown that the $\overline E$-minimum of $\psi$
cannot be attained in any point of $\overline E$, a contradiction to continuity of $\psi$ and compactness of $\overline E$.
Hence \emph{global} regular reflections cannot be strong-type in the cases considered. 

It is natural to wonder what else the global flow may be in each case. In some cases
existence of transonic or supersonic weak-type regular reflection has been proven
\cite{chen-feldman-selfsim-journal,elling-rrefl,elling-sonic-potf}. In other cases, numerical calculations
(see \cite[Figure 6]{elling-detachment}) suggest that Mach reflections should arise instead,
or even a succession of a weak-type regular reflection followed by an additional Mach reflection.

\input{nes.bbl}
\end{document}

%% file: elle.tex
\usepackage{calc}

\usepackage{amsmath}
\usepackage{amsxtra}
\usepackage{amssymb}
\usepackage{amsfonts}
\usepackage{amsthm}
\usepackage{amstext}
\usepackage{amsbsy}
\usepackage{amscd}
\usepackage[sans]{dsfont}

\usepackage{color}
\usepackage{bbding}

\newcommand{\defm}[1]{\emph{#1}}
\newcommand{\subeq}[2]{\mathord{\underbrace{\mathop{#1}}_{#2}}}

\newcommand{\sign}{\operatorname{sgn}}

\newcommand{\trace}{\operatorname{tr}}

\theoremstyle{plain}

\newtheorem{lemma}{Lemma}

\newtheorem{proposition}[lemma]{Proposition}
\theoremstyle{definition}

\newtheorem{definition}[lemma]{Definition}
\theoremstyle{remark}

\def\XXint#1#2#3{{\setbox0=\hbox{$#1{#2#3}{\int}$}
\vcenter{\hbox{$#2#3$}}\kern-.5\wd0}}

\newcommand{\vv}{\vec v}

\newcommand{\BV}{\operatorname{BV}}

\newcommand{\aint}{{-\kern-5.3mm\int}}

\newcommand{\topref}[2]{\overset{\text{\eqref{#1}}}{#2}}
\newcommand{\toprefb}[3]{\overset{\text{\myeqref{#1}}}{\underset{\text{\myeqref{#2}}}{#3}}}

%% file: mylabels.tex
\newcommand{\myeqref}[1]{\eqref{#1}}
\newcommand{\myref}[1]{\ref{#1}}
\newcommand{\mylabel}[1]{\label{#1}}
\newcommand{\myeqlabel}[1]{\label{#1}}

%% file: classrr.pstex_t
\begin{picture}(0,0)%
\includegraphics{classrr.pstex}%
\end{picture}%
\setlength{\unitlength}{3947sp}%
\begingroup\makeatletter\ifx\SetFigFont\undefined%
\gdef\SetFigFont#1#2#3#4#5{%
  \reset@font\fontsize{#1}{#2pt}%
  \fontfamily{#3}\fontseries{#4}\fontshape{#5}%
  \selectfont}%
\fi\endgroup%
\begin{picture}(7260,1749)(64,-898)
\put(526,-811){\makebox(0,0)[lb]{\smash{{\SetFigFont{8}{9.6}{\rmdefault}{\mddefault}{\updefault}{\color[rgb]{0,0,0}$t<0$}%
}}}}
\put(2326,-811){\makebox(0,0)[lb]{\smash{{\SetFigFont{8}{9.6}{\rmdefault}{\mddefault}{\updefault}{\color[rgb]{0,0,0}$t>0$}%
}}}}
\put(1183,-412){\makebox(0,0)[lb]{\smash{{\SetFigFont{8}{9.6}{\rmdefault}{\mddefault}{\updefault}{\color[rgb]{0,0,0}$\theta$}%
}}}}
\put(5926,-811){\makebox(0,0)[lb]{\smash{{\SetFigFont{8}{9.6}{\rmdefault}{\mddefault}{\updefault}{\color[rgb]{0,0,0}$t>0$}%
}}}}
\put(4126,-811){\makebox(0,0)[lb]{\smash{{\SetFigFont{8}{9.6}{\rmdefault}{\mddefault}{\updefault}{\color[rgb]{0,0,0}$t>0$}%
}}}}
\put(3826,164){\makebox(0,0)[lb]{\smash{{\SetFigFont{8}{9.6}{\rmdefault}{\mddefault}{\updefault}{\color[rgb]{0,0,0}shock}%
}}}}
\put(3826,314){\makebox(0,0)[lb]{\smash{{\SetFigFont{8}{9.6}{\rmdefault}{\mddefault}{\updefault}{\color[rgb]{0,0,0}strong-type}%
}}}}
\put(4051,539){\makebox(0,0)[lb]{\smash{{\SetFigFont{14}{16.8}{\rmdefault}{\mddefault}{\updefault}{\color[rgb]{0,0,0}$\nexists$}%
}}}}
\put(6780,661){\makebox(0,0)[lb]{\smash{{\SetFigFont{6}{7.2}{\rmdefault}{\mddefault}{\updefault}{\color[rgb]{0,0,0}Mach stem}%
}}}}
\put(6694,-342){\makebox(0,0)[lb]{\smash{{\SetFigFont{6}{7.2}{\rmdefault}{\mddefault}{\updefault}{\color[rgb]{0,0,0}discontinuity}%
}}}}
\put(6691,-231){\makebox(0,0)[lb]{\smash{{\SetFigFont{6}{7.2}{\rmdefault}{\mddefault}{\updefault}{\color[rgb]{0,0,0}Contact}%
}}}}
\put(5894,608){\makebox(0,0)[lb]{\smash{{\SetFigFont{6}{7.2}{\rmdefault}{\mddefault}{\updefault}{\color[rgb]{0,0,0}triple point}%
}}}}
\put(5549,191){\makebox(0,0)[lb]{\smash{{\SetFigFont{8}{9.6}{\rmdefault}{\mddefault}{\updefault}{\color[rgb]{0,0,0}Mach}%
}}}}
\put(5549, 41){\makebox(0,0)[lb]{\smash{{\SetFigFont{8}{9.6}{\rmdefault}{\mddefault}{\updefault}{\color[rgb]{0,0,0}reflection}%
}}}}
\put(3072,-661){\makebox(0,0)[lb]{\smash{{\SetFigFont{14}{16.8}{\rmdefault}{\mddefault}{\updefault}{\color[rgb]{0,0,0}$\exists$}%
}}}}
\put(1938,586){\makebox(0,0)[lb]{\smash{{\SetFigFont{8}{9.6}{\rmdefault}{\mddefault}{\updefault}{\color[rgb]{0,0,0}Regular}%
}}}}
\put(1938,436){\makebox(0,0)[lb]{\smash{{\SetFigFont{8}{9.6}{\rmdefault}{\mddefault}{\updefault}{\color[rgb]{0,0,0}reflection}%
}}}}
\put(151,464){\makebox(0,0)[lb]{\smash{{\SetFigFont{8}{9.6}{\rmdefault}{\mddefault}{\updefault}{\color[rgb]{0,0,0}incident}%
}}}}
\put(3076,-136){\makebox(0,0)[lb]{\smash{{\SetFigFont{8}{9.6}{\rmdefault}{\mddefault}{\updefault}{\color[rgb]{0,0,0}weak-type}%
}}}}
\put(1943,  4){\makebox(0,0)[lb]{\smash{{\SetFigFont{8}{9.6}{\rmdefault}{\mddefault}{\updefault}{\color[rgb]{0,0,0}reflected}%
}}}}
\put(4613,-415){\rotatebox{65.0}{\makebox(0,0)[lb]{\smash{{\SetFigFont{6}{7.2}{\rmdefault}{\mddefault}{\updefault}{\color[rgb]{0,0,0}elliptic region}%
}}}}}
\put(901,464){\makebox(0,0)[lb]{\smash{{\SetFigFont{8}{9.6}{\rmdefault}{\mddefault}{\updefault}{\color[rgb]{0,0,0}$\vec v=0$}%
}}}}
\put(5549,341){\makebox(0,0)[lb]{\smash{{\SetFigFont{8}{9.6}{\rmdefault}{\mddefault}{\updefault}{\color[rgb]{0,0,0}Single}%
}}}}
\put(151,-61){\makebox(0,0)[lb]{\smash{{\SetFigFont{8}{9.6}{\rmdefault}{\mddefault}{\updefault}{\color[rgb]{0,0,0}$\vec v^I$}%
}}}}
\end{picture}%

%% file: superwedge.pstex_t
\begin{picture}(0,0)%
\includegraphics{superwedge.pstex}%
\end{picture}%
\setlength{\unitlength}{3947sp}%
\begingroup\makeatletter\ifx\SetFigFont\undefined%
\gdef\SetFigFont#1#2#3#4#5{%
  \reset@font\fontsize{#1}{#2pt}%
  \fontfamily{#3}\fontseries{#4}\fontshape{#5}%
  \selectfont}%
\fi\endgroup%
\begin{picture}(7224,1674)(-11,-823)
\put(3376,614){\makebox(0,0)[lb]{\smash{{\SetFigFont{8}{9.6}{\rmdefault}{\mddefault}{\updefault}{\color[rgb]{0,0,0}weak-type}%
}}}}
\put(3376,464){\makebox(0,0)[lb]{\smash{{\SetFigFont{8}{9.6}{\rmdefault}{\mddefault}{\updefault}{\color[rgb]{0,0,0}shock}%
}}}}
\put(5026,614){\makebox(0,0)[lb]{\smash{{\SetFigFont{8}{9.6}{\rmdefault}{\mddefault}{\updefault}{\color[rgb]{0,0,0}strong-type}%
}}}}
\put(5026,464){\makebox(0,0)[lb]{\smash{{\SetFigFont{8}{9.6}{\rmdefault}{\mddefault}{\updefault}{\color[rgb]{0,0,0}shock}%
}}}}
\put(1426,-661){\makebox(0,0)[lb]{\smash{{\SetFigFont{8}{9.6}{\rmdefault}{\mddefault}{\updefault}{\color[rgb]{0,0,0}$t=0$}%
}}}}
\put(3901,-661){\makebox(0,0)[lb]{\smash{{\SetFigFont{8}{9.6}{\rmdefault}{\mddefault}{\updefault}{\color[rgb]{0,0,0}$t>0$}%
}}}}
\put(6226,-661){\makebox(0,0)[lb]{\smash{{\SetFigFont{8}{9.6}{\rmdefault}{\mddefault}{\updefault}{\color[rgb]{0,0,0}$t>0$}%
}}}}
\put(2476,464){\makebox(0,0)[lb]{\smash{{\SetFigFont{14}{16.8}{\rmdefault}{\mddefault}{\updefault}{\color[rgb]{0,0,0}$\exists$}%
}}}}
\put(5851,539){\makebox(0,0)[lb]{\smash{{\SetFigFont{14}{16.8}{\rmdefault}{\mddefault}{\updefault}{\color[rgb]{0,0,0}$\nexists$}%
}}}}
\put(5551, 14){\rotatebox{10.0}{\makebox(0,0)[lb]{\smash{{\SetFigFont{6}{7.2}{\rmdefault}{\mddefault}{\updefault}{\color[rgb]{0,0,0}elliptic region}%
}}}}}
\put(226,314){\makebox(0,0)[lb]{\smash{{\SetFigFont{8}{9.6}{\rmdefault}{\mddefault}{\updefault}{\color[rgb]{0,0,0}$\vec v^I$}%
}}}}
\end{picture}%

%% file: locrr.pstex_t
\begin{picture}(0,0)%
\includegraphics{locrr.pstex}%
\end{picture}%
\setlength{\unitlength}{3947sp}%
\begingroup\makeatletter\ifx\SetFigFont\undefined%
\gdef\SetFigFont#1#2#3#4#5{%
  \reset@font\fontsize{#1}{#2pt}%
  \fontfamily{#3}\fontseries{#4}\fontshape{#5}%
  \selectfont}%
\fi\endgroup%
\begin{picture}(3024,1149)(-11,-298)
\put(2456,526){\makebox(0,0)[lb]{\smash{{\SetFigFont{6}{7.2}{\rmdefault}{\mddefault}{\updefault}{\color[rgb]{0,0,0}shock}%
}}}}
\put(2456,643){\makebox(0,0)[lb]{\smash{{\SetFigFont{6}{7.2}{\rmdefault}{\mddefault}{\updefault}{\color[rgb]{0,0,0}Incident}%
}}}}
\put(451,-136){\makebox(0,0)[lb]{\smash{{\SetFigFont{6}{7.2}{\rmdefault}{\mddefault}{\updefault}{\color[rgb]{0,0,0}solid}%
}}}}
\put(1351,-61){\makebox(0,0)[lb]{\smash{{\SetFigFont{6}{7.2}{\rmdefault}{\mddefault}{\updefault}{\color[rgb]{0,0,0}reflection}%
}}}}
\put(1351,-181){\makebox(0,0)[lb]{\smash{{\SetFigFont{6}{7.2}{\rmdefault}{\mddefault}{\updefault}{\color[rgb]{0,0,0}point}%
}}}}
\put(376,614){\makebox(0,0)[lb]{\smash{{\SetFigFont{6}{7.2}{\rmdefault}{\mddefault}{\updefault}{\color[rgb]{0,0,0}Reflected}%
}}}}
\put(676,464){\makebox(0,0)[lb]{\smash{{\SetFigFont{6}{7.2}{\rmdefault}{\mddefault}{\updefault}{\color[rgb]{0,0,0}shock}%
}}}}
\put(400,110){\makebox(0,0)[lb]{\smash{{\SetFigFont{6}{7.2}{\rmdefault}{\mddefault}{\updefault}{\color[rgb]{0,0,0}$\vec v_3$}%
}}}}
\put(2492,109){\makebox(0,0)[lb]{\smash{{\SetFigFont{6}{7.2}{\rmdefault}{\mddefault}{\updefault}{\color[rgb]{0,0,0}$\vec v_1$}%
}}}}
\put(1452,498){\makebox(0,0)[lb]{\smash{{\SetFigFont{6}{7.2}{\rmdefault}{\mddefault}{\updefault}{\color[rgb]{0,0,0}$\vec v_2$}%
}}}}
\end{picture}%

%% file: polar.pstex_t
\begin{picture}(0,0)%
\includegraphics{polar.pstex}%
\end{picture}%
\setlength{\unitlength}{3947sp}%
\begingroup\makeatletter\ifx\SetFigFont\undefined%
\gdef\SetFigFont#1#2#3#4#5{%
  \reset@font\fontsize{#1}{#2pt}%
  \fontfamily{#3}\fontseries{#4}\fontshape{#5}%
  \selectfont}%
\fi\endgroup%
\begin{picture}(3024,2724)(-11,-1873)
\put(2383,130){\makebox(0,0)[lb]{\smash{{\SetFigFont{6}{7.2}{\rmdefault}{\mddefault}{\updefault}{\color[rgb]{0,0,0}$\tau$}%
}}}}
\put(2128,266){\makebox(0,0)[lb]{\smash{{\SetFigFont{6}{7.2}{\rmdefault}{\mddefault}{\updefault}{\color[rgb]{0,0,0}S}%
}}}}
\put(1243,272){\makebox(0,0)[lb]{\smash{{\SetFigFont{6}{7.2}{\rmdefault}{\mddefault}{\updefault}{\color[rgb]{0,0,0}W}%
}}}}
\put(613,269){\makebox(0,0)[lb]{\smash{{\SetFigFont{6}{7.2}{\rmdefault}{\mddefault}{\updefault}{\color[rgb]{0,0,0}U}%
}}}}
\put(2385,-37){\makebox(0,0)[lb]{\smash{{\SetFigFont{6}{7.2}{\rmdefault}{\mddefault}{\updefault}{\color[rgb]{0,0,0}$\tau_*$}%
}}}}
\put(1473,-379){\makebox(0,0)[lb]{\smash{{\SetFigFont{6}{7.2}{\rmdefault}{\mddefault}{\updefault}{\color[rgb]{0,0,0}$\vec v_2$}%
}}}}
\put(1158,-664){\makebox(0,0)[lb]{\smash{{\SetFigFont{6}{7.2}{\rmdefault}{\mddefault}{\updefault}{\color[rgb]{0,0,0}Shock}%
}}}}
\put(1158,-784){\makebox(0,0)[lb]{\smash{{\SetFigFont{6}{7.2}{\rmdefault}{\mddefault}{\updefault}{\color[rgb]{0,0,0}polar}%
}}}}
\put(1281,638){\makebox(0,0)[lb]{\smash{{\SetFigFont{6}{7.2}{\rmdefault}{\mddefault}{\updefault}{\color[rgb]{0,0,0}Shock normal}%
}}}}
\put(1573,148){\makebox(0,0)[lb]{\smash{{\SetFigFont{6}{7.2}{\rmdefault}{\mddefault}{\updefault}{\color[rgb]{0,0,0}$\vec v_3$}%
}}}}
\put(2223,-76){\makebox(0,0)[lb]{\smash{{\SetFigFont{6}{7.2}{\rmdefault}{\mddefault}{\updefault}{\color[rgb]{0,0,0}N}%
}}}}
\end{picture}%

%% file: bigger90.pstex_t
\begin{picture}(0,0)%
\includegraphics{bigger90.pstex}%
\end{picture}%
\setlength{\unitlength}{3947sp}%
\begingroup\makeatletter\ifx\SetFigFont\undefined%
\gdef\SetFigFont#1#2#3#4#5{%
  \reset@font\fontsize{#1}{#2pt}%
  \fontfamily{#3}\fontseries{#4}\fontshape{#5}%
  \selectfont}%
\fi\endgroup%
\begin{picture}(5874,1974)(-11,-4273)
\put(174,-3649){\makebox(0,0)[lb]{\smash{{\SetFigFont{6}{7.2}{\rmdefault}{\mddefault}{\updefault}{\color[rgb]{0,0,0}$t=0$}%
}}}}
\put(1047,-3840){\makebox(0,0)[lb]{\smash{{\SetFigFont{6}{7.2}{\rmdefault}{\mddefault}{\updefault}{\color[rgb]{0,0,0}$\alpha$}%
}}}}
\put( 76,-2986){\makebox(0,0)[lb]{\smash{{\SetFigFont{6}{7.2}{\rmdefault}{\mddefault}{\updefault}{\color[rgb]{0,0,0}Incident}%
}}}}
\put(2054,-2768){\makebox(0,0)[lb]{\smash{{\SetFigFont{6}{7.2}{\rmdefault}{\mddefault}{\updefault}{\color[rgb]{0,0,0}Incident}%
}}}}
\put(3101,-2682){\makebox(0,0)[lb]{\smash{{\SetFigFont{6}{7.2}{\rmdefault}{\mddefault}{\updefault}{\color[rgb]{0,0,0}either type}%
}}}}
\put(3976,-2836){\makebox(0,0)[lb]{\smash{{\SetFigFont{6}{7.2}{\rmdefault}{\mddefault}{\updefault}{\color[rgb]{0,0,0}Incident}%
}}}}
\put(2176,-3436){\makebox(0,0)[lb]{\smash{{\SetFigFont{6}{7.2}{\rmdefault}{\mddefault}{\updefault}{\color[rgb]{0,0,0}$t>0$}%
}}}}
\put(4126,-3436){\makebox(0,0)[lb]{\smash{{\SetFigFont{6}{7.2}{\rmdefault}{\mddefault}{\updefault}{\color[rgb]{0,0,0}$t>0$}%
}}}}
\put(5051,-2682){\makebox(0,0)[lb]{\smash{{\SetFigFont{6}{7.2}{\rmdefault}{\mddefault}{\updefault}{\color[rgb]{0,0,0}weak-type only}%
}}}}
\put(1662,-3885){\makebox(0,0)[lb]{\smash{{\SetFigFont{6}{7.2}{\rmdefault}{\mddefault}{\updefault}{\color[rgb]{0,0,0}$\theta$}%
}}}}
\end{picture}%

%% file: clarr.pstex_t
\begin{picture}(0,0)%
\includegraphics{clarr.pstex}%
\end{picture}%
\setlength{\unitlength}{3947sp}%
\begingroup\makeatletter\ifx\SetFigFont\undefined%
\gdef\SetFigFont#1#2#3#4#5{%
  \reset@font\fontsize{#1}{#2pt}%
  \fontfamily{#3}\fontseries{#4}\fontshape{#5}%
  \selectfont}%
\fi\endgroup%
\begin{picture}(6549,2424)(214,-1273)
\put(2926,689){\makebox(0,0)[lb]{\smash{{\SetFigFont{8}{9.6}{\rmdefault}{\mddefault}{\updefault}{\color[rgb]{0,0,0}incoming}%
}}}}
\put(2926,539){\makebox(0,0)[lb]{\smash{{\SetFigFont{8}{9.6}{\rmdefault}{\mddefault}{\updefault}{\color[rgb]{0,0,0}shock}%
}}}}
\put(2712,-397){\makebox(0,0)[lb]{\smash{{\SetFigFont{8}{9.6}{\rmdefault}{\mddefault}{\updefault}{\color[rgb]{0,0,0}$\vec v\rightarrow 0$}%
}}}}
\put(956,-1120){\makebox(0,0)[lb]{\smash{{\SetFigFont{8}{9.6}{\rmdefault}{\mddefault}{\updefault}{\color[rgb]{0,0,0}$\vec n_B$}%
}}}}
\put(976,-286){\makebox(0,0)[lb]{\smash{{\SetFigFont{8}{9.6}{\rmdefault}{\mddefault}{\updefault}{\color[rgb]{0,0,0}$B$}%
}}}}
\put(2326,-136){\makebox(0,0)[lb]{\smash{{\SetFigFont{8}{9.6}{\rmdefault}{\mddefault}{\updefault}{\color[rgb]{0,0,0}$A$}%
}}}}
\put(2401,-811){\makebox(0,0)[lb]{\smash{{\SetFigFont{8}{9.6}{\rmdefault}{\mddefault}{\updefault}{\color[rgb]{0,0,0}$\vec n_A$}%
}}}}
\put(1426,-211){\makebox(0,0)[lb]{\smash{{\SetFigFont{8}{9.6}{\rmdefault}{\mddefault}{\updefault}{\color[rgb]{0,0,0}region $E$}%
}}}}
\put(1726,614){\makebox(0,0)[rb]{\smash{{\SetFigFont{8}{9.6}{\rmdefault}{\mddefault}{\updefault}{\color[rgb]{0,0,0}reflected shock $S$}%
}}}}
\put(3301,164){\makebox(0,0)[lb]{\smash{{\SetFigFont{8}{9.6}{\rmdefault}{\mddefault}{\updefault}{\color[rgb]{0,0,0}$\vec\xi_r$}%
}}}}
\put(1426,989){\makebox(0,0)[lb]{\smash{{\SetFigFont{8}{9.6}{\rmdefault}{\mddefault}{\updefault}{\color[rgb]{0,0,0}region $I$}%
}}}}
\put(2476,614){\makebox(0,0)[lb]{\smash{{\SetFigFont{8}{9.6}{\rmdefault}{\mddefault}{\updefault}{\color[rgb]{0,0,0}$\vec v^I$}%
}}}}
\put(6001,989){\makebox(0,0)[lb]{\smash{{\SetFigFont{8}{9.6}{\rmdefault}{\mddefault}{\updefault}{\color[rgb]{0,0,0}incoming}%
}}}}
\put(6007,883){\makebox(0,0)[lb]{\smash{{\SetFigFont{8}{9.6}{\rmdefault}{\mddefault}{\updefault}{\color[rgb]{0,0,0}shock}%
}}}}
\put(3526,689){\makebox(0,0)[lb]{\smash{{\SetFigFont{8}{9.6}{\rmdefault}{\mddefault}{\updefault}{\color[rgb]{0,0,0}Strong}%
}}}}
\put(3526,539){\makebox(0,0)[lb]{\smash{{\SetFigFont{8}{9.6}{\rmdefault}{\mddefault}{\updefault}{\color[rgb]{0,0,0}shock}%
}}}}
\put(3526,389){\makebox(0,0)[lb]{\smash{{\SetFigFont{8}{9.6}{\rmdefault}{\mddefault}{\updefault}{\color[rgb]{0,0,0}tangent}%
}}}}
\put(4493,980){\makebox(0,0)[lb]{\smash{{\SetFigFont{8}{9.6}{\rmdefault}{\mddefault}{\updefault}{\color[rgb]{0,0,0}$\vec v^I$}%
}}}}
\end{picture}%

%% file: super2.pstex_t
\begin{picture}(0,0)%
\includegraphics{super2.pstex}%
\end{picture}%
\setlength{\unitlength}{3947sp}%
\begingroup\makeatletter\ifx\SetFigFont\undefined%
\gdef\SetFigFont#1#2#3#4#5{%
  \reset@font\fontsize{#1}{#2pt}%
  \fontfamily{#3}\fontseries{#4}\fontshape{#5}%
  \selectfont}%
\fi\endgroup%
\begin{picture}(2348,3059)(172,-1997)
\put(1679,421){\makebox(0,0)[lb]{\smash{{\SetFigFont{8}{9.6}{\rmdefault}{\mddefault}{\updefault}{\color[rgb]{0,0,0}strong-type}%
}}}}
\put(1367,818){\makebox(0,0)[lb]{\smash{{\SetFigFont{8}{9.6}{\rmdefault}{\mddefault}{\updefault}{\color[rgb]{0,0,0}$\vec v_I$}%
}}}}
\put(1675,276){\makebox(0,0)[lb]{\smash{{\SetFigFont{8}{9.6}{\rmdefault}{\mddefault}{\updefault}{\color[rgb]{0,0,0}shock tangent}%
}}}}
\put(1741,-1131){\makebox(0,0)[lb]{\smash{{\SetFigFont{8}{9.6}{\rmdefault}{\mddefault}{\updefault}{\color[rgb]{0,0,0}arc $P$}%
}}}}
\put(1008,-45){\rotatebox{300.0}{\makebox(0,0)[lb]{\smash{{\SetFigFont{8}{9.6}{\rmdefault}{\mddefault}{\updefault}{\color[rgb]{0,0,0}elliptic region $E$}%
}}}}}
\put(281, 70){\makebox(0,0)[lb]{\smash{{\SetFigFont{8}{9.6}{\rmdefault}{\mddefault}{\updefault}{\color[rgb]{0,0,0}$\vxr$}%
}}}}
\put(1848,-1829){\makebox(0,0)[lb]{\smash{{\SetFigFont{8}{9.6}{\rmdefault}{\mddefault}{\updefault}{\color[rgb]{0,0,0}$\vec v_H$}%
}}}}
\put(1967,-1444){\makebox(0,0)[lb]{\smash{{\SetFigFont{8}{9.6}{\rmdefault}{\mddefault}{\updefault}{\color[rgb]{0,0,0}$H$}%
}}}}
\put(202,884){\makebox(0,0)[lb]{\smash{{\SetFigFont{8}{9.6}{\rmdefault}{\mddefault}{\updefault}{\color[rgb]{0,0,0}$\vec v\rightarrow 0$}%
}}}}
\put(1701,-244){\makebox(0,0)[lb]{\smash{{\SetFigFont{8}{9.6}{\rmdefault}{\mddefault}{\updefault}{\color[rgb]{0,0,0}$S$}%
}}}}
\end{picture}%

%% file: corner.pstex_t
\begin{picture}(0,0)%
\includegraphics{corner.pstex}%
\end{picture}%
\setlength{\unitlength}{3947sp}%
\begingroup\makeatletter\ifx\SetFigFont\undefined%
\gdef\SetFigFont#1#2#3#4#5{%
  \reset@font\fontsize{#1}{#2pt}%
  \fontfamily{#3}\fontseries{#4}\fontshape{#5}%
  \selectfont}%
\fi\endgroup%
\begin{picture}(2077,2517)(1614,-1713)
\put(1885,-1417){\makebox(0,0)[lb]{\smash{{\SetFigFont{9}{10.8}{\rmdefault}{\mddefault}{\updefault}{\color[rgb]{0,0,0}$\theta$}%
}}}}
\put(1758,-897){\makebox(0,0)[lb]{\smash{{\SetFigFont{9}{10.8}{\rmdefault}{\mddefault}{\updefault}{\color[rgb]{0,0,0}$\vxr$}%
}}}}
\put(2348,-1106){\makebox(0,0)[lb]{\smash{{\SetFigFont{9}{10.8}{\rmdefault}{\mddefault}{\updefault}{\color[rgb]{0,0,0}$E$}%
}}}}
\put(2322,668){\makebox(0,0)[rb]{\smash{{\SetFigFont{8}{9.6}{\rmdefault}{\mddefault}{\updefault}{\color[rgb]{0,0,0}strong-type}%
}}}}
\put(3061,668){\makebox(0,0)[lb]{\smash{{\SetFigFont{8}{9.6}{\rmdefault}{\mddefault}{\updefault}{\color[rgb]{0,0,0}weak-type}%
}}}}
\put(3193,-232){\makebox(0,0)[lb]{\smash{{\SetFigFont{9}{10.8}{\rmdefault}{\mddefault}{\updefault}{\color[rgb]{0,0,0}$\theta$}%
}}}}
\put(2901, 17){\makebox(0,0)[lb]{\smash{{\SetFigFont{9}{10.8}{\rmdefault}{\mddefault}{\updefault}{\color[rgb]{0,0,0}$\alpha$}%
}}}}
\put(1726,-250){\makebox(0,0)[lb]{\smash{{\SetFigFont{9}{10.8}{\rmdefault}{\mddefault}{\updefault}{\color[rgb]{0,0,0}$\vec v_I$}%
}}}}
\put(2429,357){\makebox(0,0)[lb]{\smash{{\SetFigFont{8}{9.6}{\rmdefault}{\mddefault}{\updefault}{\color[rgb]{0,0,0}$-g_{\vec v}$}%
}}}}
\put(3080,-1443){\makebox(0,0)[lb]{\smash{{\SetFigFont{9}{10.8}{\rmdefault}{\mddefault}{\updefault}{\color[rgb]{0,0,0}$\vec z$}%
}}}}
\put(3334,-942){\makebox(0,0)[lb]{\smash{{\SetFigFont{9}{10.8}{\rmdefault}{\mddefault}{\updefault}{\color[rgb]{0,0,0}$\vn_r$}%
}}}}
\put(3330,118){\makebox(0,0)[lb]{\smash{{\SetFigFont{9}{10.8}{\rmdefault}{\mddefault}{\updefault}{\color[rgb]{0,0,0}$\vt_r$}%
}}}}
\put(2445,-645){\makebox(0,0)[lb]{\smash{{\SetFigFont{9}{10.8}{\rmdefault}{\mddefault}{\updefault}{\color[rgb]{0,0,0}$T$}%
}}}}
\end{picture}%